\documentclass[times]{amsart}

\usepackage{geometry}
\usepackage{todonotes}
\usepackage{xcolor, soul}
\usepackage[colorlinks=true,linkcolor=black,anchorcolor=black,citecolor=black,filecolor=black,menucolor=black,runcolor=black,urlcolor=black]{hyperref}

\newgeometry{left=2.5cm, right=2.5cm, top=2cm, bottom=2cm}
\usepackage{babel}
\theoremstyle{definition}
\newtheorem{theorem}{Theorem}[section]
\newtheorem{lemma}[theorem]{Lemma}
\newtheorem{proposition}[theorem]{Proposition}
\newtheorem{corollary}[theorem]{Corollary}
\newtheorem{remark}[theorem]{Remark}

\newcommand{\ip}[2]{{\langle#1},{#2\rangle }}
\newcommand{\der}[1]{\mathrm{#1}}

\newcommand{\R}{{\mathbb R}}
\newcommand{\C}{{\mathbb C}}
\newcommand{\F}{{\mathcal F}}

\newcommand{\norm}[1]{\left\lVert#1\right\rVert}

\begin{document}
\title{Convexity on manifolds without focal points and applications}
\author{Aprameyan Parthasarathy and B Sivashankar}

\address{Department of Mathematics, Indian Institute of Technology, Madras, India}
\email{aprameyan@iitm.ac.in, Sivashankar.B@outlook.com}

\begin{abstract}
   
    \small{ In this article, we study strictly convex functions on Riemannian manifolds without focal points, a broad class of manifolds encompassing all Hadamard manifolds as well as a large collection of manifolds whose sectional curvatures change sign. Using geometrically defined convex functions on such manifolds, we derive interesting consequences such as the continuity of the isoperimetric profile function without conditions on the sectional curvatures; if the manifolds are also K\"ahler, we obtain Steinness as well as a lower bound on the volume growth of metric balls. Our primary applications concern the spectrum of the Laplacian. We prove that the absolutely continuous part of the spectrum contains a certain infinite interval assuming only the existence of a point with respect to which the radial curvatures are nonpositive. This yields a generalization of the corresponding result for Hadamard manifolds. We use the geometry at infinity to give a new construction of a strictly convex function. We then apply this to show that the spectrum is purely absolutely continuous on a class of manifolds for which the horospheres in every direction at a single point have constant mean curvatures (e.g. asymptotically harmonic manifolds, symmetric spaces of noncompact type). Finally, we show the equality of Cheeger's constant and the volume entropy for a broad class of manifolds.}
    
\end{abstract}


\maketitle

\section{Introduction}\label{sec1}

     \noindent Convex functions have long played a fundamental role in Riemannian geometry. A  real-valued function on a Riemannian manifold is (strictly) convex if its Hessian is positive (definite) semidefinite everywhere. Under appropriate curvature conditions there are natural geometrically defined convex functions on a Riemannian manifold; profound structural results arise from the interaction of curvature and convexity. On simply connected Riemannian manifolds of nonpositive sectional curvature, called Hadamard manifolds, the square of the distance function from a fixed but arbitrary point and the Busemann functions are examples of convex functions which play a crucial role in the geometry of these manifolds. While the former is a smooth strictly convex function, the latter is merely a $C^2$- convex function \cite[\S 1.6.6]{Eberlein:1996}. In the nonnegative sectional curvature situation for instance, a key idea in the fundamental work \cite{CheegerGromoll:1972} is the construction of suitable continuous convex functions related to Busemann functions and having appropriate additional properties. Indeed, the mere existence or nonexistence of convex functions have geometric and topological implications. Nonconstant convex functions on complete Riemannian manifolds exist only on those with infinite volume and so they don't exist on compact Riemannian manifolds. If a Riemannian manifold admits a strictly convex function, then it cannot contain compact immersed minimal submanifolds and hence no closed geodesics. In fact, in such a situation, each geodesic leaves every compact subset after a finite time. Here we study the effect of strict convexity on the quantum analogue for this classical phenomenon captured by the absolute continuity of the spectrum of the Laplacian. This is of primary interest from the point of view of asymptotic completeness and quantum mechanical scattering \cite[\S XIII.6, XIII.7]{ReedSimon:1978}. Thus our main goal is to understand the nature of the spectrum of the Laplacian vis-\`a-vis the absence of point spectrum (bound states) and the singularly continuous spectrum. We establish this for a large class of manifolds without assumptions on the sign of the sectional curvatures. \\
     
    \noindent The study of the spectrum of the Laplacian on complete Riemannian manifolds has a long, rich history and there are very many results, often with various additional constraints placed on the structure of the manifold and on its curvature. For instance, in the case of manifolds of nonpositive curvature the essential spectrum was shown to be an infinite interval depending on the dimension and sectional curvatures at infinity in \cite[Theorem 6.3]{Donnelly:Essential}, using the distance function. One can then ask interesting questions about the finer structure of the spectrum e.g., whether there are eigenvalues in the essential spectrum. For a succinct recent overview of some of these aspects, we refer to \cite{Donnelly:2010}. We now content ourselves by mentioning a few results which use convexity in an essential manner to deduce refined information on the nature of the spectrum and to put our work in context, with no pretensions to completeness. In the case of nonnegative sectional curvatures, explicit convex functions have also been used to show that there are no $\mathrm{L}^2$-eigenfunctions under the additional assumptions on the codimension of the soul of the manifold being $1$ or having a quadratic decay of radial curvatures when the soul is just a point \cite[Theorems 4.2, 4.3]{Escobar:1992}. On Hadamard manifolds, \cite[Theorem 1]{Xavier:1988} exhibited a strictly convex function related to the square of the distance function satisfying the estimates required to ensure the nontriviality of the absolutely continuous spectrum. However, this is not sufficient to exclude the existence either of eigenvalues or of the  singularly continuous spectrum. In what follows, we show that the absolute continuity of the spectrum holds for several interesting classes of manifolds. \\

    \noindent In this article, we consider convex functions on complete Riemannian manifolds $M$ whose Jacobi fields satisfy certain growth properties, while making no assumptions on sectional curvatures. A simply connected Riemannian manifold is said to have no focal points if the norm of any nontrivial Jacobi field  which vanishes at the initial point of a geodesic ray is strictly increasing. Since the exponential map at any point on such a manifold is a diffeomorphism, i.e., every point is a pole, they are necessarily noncompact. This class contains all Hadamard manifolds; however, it is well known that there are manifolds without focal points whose sectional curvatures change sign \cite{Gulliver:1975}. While several aspects of the geometry of these manifolds have been well-studied, from understanding the dynamics of the geodesic flow (\cite{Eberlein:1972}, \cite{Sullivan:1974}, \cite{Eschenburg:1977}), we study some of these using the convexity of certain geometrically defined functions on them.\\
    
    \noindent The square of the distance function from an arbitrary but fixed point on such manifolds is smooth and strictly convex (\cite{Eberlein:1972}). It is suitably modified to ensure a bounded gradient, an aspect that is pivotal in many applications. For example, our first application concerns the continuity of the isoperimetric profile function $\mathcal{I}_M$ defined on the manifold $M$. This function maps every positive volume to the infimum of the perimeters of sets in $M$ having that volume. Examples are discussed under sectional curvature conditions in \cite{Ritore:2017}, namely, Hadamard manifolds and manifolds with strictly positive sectional curvatures. We observe in Proposition \ref{isoperimetric profile} that for simply connected Riemannian manifolds without focal points the map $\mathcal I_M$ is a nondecreasing continuous function, thus giving a class of examples without constraints on sectional curvatures. Strictly convex functions on K\"ahler manifolds can be seen to be strictly plurisubharmonic functions. So, manifolds without focal points which are also K\"ahler, admit strictly plurisubharmonic functions. From this, we observe in Proposition \ref{Stein} that such manifolds are Stein using Grauert's proof of Levi's problem \cite[\S 4]{GreeneWu:1977}. By composing with a suitable strictly convex function on $\mathbb R$, we get a strictly convex exhaustion function with bounded gradient. This is applied in Proposition \ref{volume growth} to obtain a lower bound for the growth of the volume of metric balls following \cite[Lemma 4.2 (i)]{NiTam:2003} without imposing any bounds on curvature.  \\

    \noindent We now proceed to describe our first spectral-theoretic result. A key idea is due to \cite{Xavier:1988} showing that on complete Riemannian manifolds supporting a sufficiently regular strictly convex function, Kato's theory of $H$- smooth operators can be applied to understand the spectrum of the Laplacian. In the case of Hadamard manifolds, this was done in \cite[Theorem 1]{Xavier:1988} under suitable bounds on the Laplacian and its derivatives to conclude that the absolutely continuous part of the spectrum was nontrivial. In fact, it contains an interval that depends explicitly on these bounds. In the present work, we extend the implementation of $H$- smoothing to simply connected Riemannian manifolds without focal points  in Theorem \ref{radial absolutely continuous}. Assuming only the nonpositivity of radial curvatures with respect to a point and bounds on Laplacian and the bi-Laplacian, we show that the absolutely continuous spectrum is an infinite interval which depends on the bounds. This is carried out by applying Lemma \ref{hessian bounds} where we obtain the required quantitative estimates on the Hessian of distance function when the manifold necessarily has no focal points.  For context, we mention some of the work that derive information about the spectrum from the behaviour of the radial curvatures. Absence of point spectrum has been shown for manifolds with a pole satisfying quadratic decay of radial curvatures \cite[Theorem 1.5]{Donnelly:2010}. Under linear decay of radial curvatures it was proved in \cite{Kumura:2010} that there are no eigenvalues embedded in the essential spectrum.\\
    
    \noindent In view of Theorem \ref{radial absolutely continuous}, we can ask for more refined properties of the spectrum such as identifying whether the point spectrum or the singularly continuous spectrum are nontrivial. A sufficient condition for the absence of point spectrum and the singularly continuous spectrum is given by the existence of a strictly convex function whose bi-Laplacian is nonpositive \cite[Corollary to Theorem 2]{Xavier:1988}. As noted there, this seemingly unnatural additional condition is, in fact, necessary. Indeed, if this condition on the bi-Laplacian fails, Hadamard manifolds with eigenvalues in the spectrum are constructed. This sufficiency condition seems to us not to be approachable by the use of distance function. \\

    \noindent To overcome this, we now turn our attention to the Busemann functions associated to unit vectors in the tangent bundle of $M$. On simply connected Riemmanian manifolds without focal points, these are $C^2$- convex \cite[Theorem 2]{Eschenburg:1977} but not strictly convex. We remedy the situation by giving a new construction of a strictly convex function from Busemann functions. This was motivated by the idea outlined in the context of symmetric spaces of noncompact type in \cite[p. $657$]{DonnellyXavier:2006}. Using the geometry at infinity we construct a strictly convex function in Theorem  \ref{strictly convex function} by averaging over Busemann functions. A careful analysis of Jacobi tensors and its relation to the Hessian of Busemann functions is central to these considerations. As  principal applications of Theorem \ref{strictly convex function}, we show that the spectrum $\sigma(\Delta)$ of the Laplacian $\Delta$ is purely absolutely continuous (i.e., there are no $\mathrm{L}^2$-eigenvalues and there is no singularly continuous spectrum) for an interesting class of manifolds. These are simply connected Riemannian manifolds without focal points in which all the horospheres have constant mean curvatures in every direction at a single point. This include all well-known examples such as Damek-Ricci spaces or more generally harmonic manifolds without focal points as well as Riemannian symmetric spaces of noncompact type and higher rank. We give an elegant proof for the equality of spectrum  in the context of harmonic manifolds without focal points by making use of the radial Fourier transform (\cite{Biswas:2021}) and certain estimates on the $\mathrm{c}$-function related to the Plancherel measure.\\

    \noindent In the last section, we show that the two fundamental geometric invariants namely, Cheeger's constant and the volume entropy coincide for a broad class of manifolds for which the horospheres in every direction at a single point have nonnegative constant mean curvatures  (Theorem \ref{cheeger and bottom}). In particular, equality holds in Cheeger's inequality for the bottom of the spectrum, expressed in terms of these constant mean curvatures. This generalizes the results obtained for asymptotically harmonic manifolds which are with bounded asymptote \cite[Theorem 5.4]{KnieperPeyerimhoff:2015} and that of symmetric spaces of noncompact type \cite[Theorem 1]{Wang:2015}. 
    
	\section{Preliminaries}
	
    \noindent In this section, we give a brief summary of the basic geometric setup needed in the sequel. Let $M$ be a complete Riemannian manifold and let $\gamma$ be a geodesic in $M$. A vector field $Y$ along $\gamma$ is called a Jacobi field if it satisfies the Jacobi equation
	\begin{equation} \label{Jacobi field equation}
		Y''+R (\gamma',Y)\gamma'=0.
	\end{equation}
	Here $\gamma'$ is the tangent vector field of $\gamma$, $R$ is the Riemann curvature tensor on $M$ and $Y'$ denotes the covariant derivative of $Y$  with respect to $\gamma'$ along $\gamma$. The prototypical example of a Jacobi field is the variation vector field of a geodesic variation. Two points on $\gamma$ are said to be conjugate along $\gamma$ if there exists a nonzero Jacobi field which vanishes at both the points. $M$ is called  a manifold without conjugate points if for every geodesic $\gamma$ in $M$, no two points on $\gamma$ are conjugate along $\gamma$ or equivalently, if the exponential map is non-singular at every point in $M$. In fact, $M$ has no conjugate points if and only if $\exp_p: T_pM  \to M $ is a covering map for each $p$ in $M$. So if $M$ is also simply connected, then $\exp_p$ is a diffeomorphism for all $p\in M$ \cite[Theorem 1]{Sullivan:1974} and hence $M$ is necessarily noncompact. In this case, any two distinct points of $M$ are joined by a unique geodesic. \\
	
    \noindent A simply connected Riemannian manifold $M$ is said to have no focal points if for every nontrivial Jacobi field $Y$ satisfying $Y(0)=0$, the map $t \mapsto \norm{Y(t)}$ is strictly increasing for $t>0$ \cite[Proposition 4]{Sullivan:1974}. In particular, any nontrivial Jacobi field cannot vanish more than once and hence every manifold without focal points is also a manifold without conjugate points. Note that all Hadamard manifolds are manifolds without focal points since $t \mapsto \norm{Y(t)}$ is strictly convex \cite[Theorem 4]{Sullivan:1974}. \\
   
    \noindent For some applications, it is not just convenient but even important to consider the full vector space of normal Jacobi fields along a unit speed geodesic together rather than consider them one at a time. A Jacobi tensor is the tool which does this. We now give a brief overview of Jacobi tensors and refer to \cite{Eschenburg:1977} for a full account. Let $M$ be a simply connected Riemannian manifold without conjugate points and let ($SM$) $TM$ denote its (unit) tangent bundle. For $v \in SM$, let $\gamma_{_v}$ denote the unit speed geodesic determined by $v$. A normal $(1,1)$- tensor $Y_v$ along the geodesic $\gamma_{_v}$ is a bundle endomorphism of the normal bundle $N\gamma_{_v}$ over $\gamma_{_v}$ given by 
	\begin{align*}
		Y_v:\R \to \mathrm{End}N(\gamma_{_v})=\bigcup\limits_{t \in \R}\mathrm{End}(N(\gamma_{_v}(t))),\\
         N(\gamma_{_v}(t)):=\left\{w \in T_{\gamma_{_v}(t)}M:\ip{w}{\gamma_{_v}'(t)}_{\gamma_{_v}(t)}=0 \right\} . 
	\end{align*}
    A Jacobi tensor $Y_v$ is a normal $(1,1)$- tensor field along a geodesic $\gamma_{_v}$ with transversal derivative (i.e., there exists some $t$ such that $ker~Y_v(t) \cap ker~Y_{v}'(t)=0$) satisfying the Jacobi equation
	\begin{equation}\label{Jacobi tensor equation}
		Y_{v}''(t)+R(t) Y_v(t)=0.
	\end{equation}
	Here $R(t)$ is the symmetric $(1,1)$- tensor along $\gamma_{_v}$ induced by the curvature tensor $R$ and defined by
	\begin{equation*}
		R(t)w:=R(\gamma_{_v}'(t),w)\gamma_{_v}'(t), \quad w \in T_{\gamma_{_v}(t)}M
	\end{equation*}
	and $Y_{v}'$ is the covariant derivative of the tensor $Y_v$ along $\gamma_{_v}$ with respect to the tangent vector field $\gamma_{_v}'$. \\

    \noindent Lagrange tensors form an important subspace of the space of Jacobi tensors along a geodesic $\gamma_{_v}$, consisting of those Jacobi tensors $A_v$ whose Wronskian $W(A_v, A_v)$ vanishes \cite[\S 1]{Eschenburg:1977}.  If we have a nowhere singular Lagrange tensor $A_v$, then we can compute each Jacobi tensor $Y_v$ along the same geodesic $\gamma_{_v}$ from its initial or boundary values. If $\gamma_{_v}$ has no conjugate points, then the Lagrange tensor $A_v(t)$ given by the initial conditions $A_v(0)=0$ and $A_{v}'(0)=\mathrm{Id}$ is nonsingular for all $t \neq 0$ and determines the shape tensor of geodesic sphere centred at $\gamma_{_v}(0)$ at the point $\gamma_{_v}(t)$ .  \\
    
    \noindent For each $s >0$ and $v \in SM$, we also consider the Lagrange tensors $D_{s,v}$ associated to geodesic spheres centred at $\gamma_{_v}(s)$. These are solutions to the Jacobi equation \eqref{Jacobi tensor equation} with the boundary conditions $D_{s,v}(0)=\mathrm{Id}$ and $D_{s,v}(s)=0$. Then $D_v:=\lim\limits_{s \to \infty}D_{s,v}$ exists and is referred to as the stable Jacobi tensor $D_v$ satisfying $D_v (0)=\mathrm{Id}$. A simply connected Riemannian manifold without conjugate points is said to be a manifold of continuous asymptote if $v \to D_{v}'(0)$ is a continuous section of the vector bundle $\{(v,\phi): v \in SM, \phi \in \mathrm{End}(N(\gamma_{_v})(0)) \}$ over $SM$.
    For such manifolds, \cite[\S 11]{KnieperPeyerimhoff:2013} gives a compactification known as the Busemann compactification. Let $o \in M$ be an arbitrary but fixed point. Then the Busemann boundary $\partial^o M$ is homeomorphic to $S_pM$ for any $p \in M$.\\ 

    \noindent Convexity being the theme of this article, the two natural geometrically defined functions which under suitable constraints on curvature are convex are the distance function and the Busemann function defined below. Given $t\in [0, \infty)$ and $v \in SM$, let $b_{t, v}:M\to \mathbb R$ denote the function defined by
	\begin{equation*}
		b_{t,v}(x):=d(x,\gamma_{_v}(t)) - t
	\end{equation*}
	which, by triangle inequality, is increasing as a function of $t$. Since $b_{t,v}(\cdot)$ is absolutely bounded by $d(\gamma_{_v}(0),\cdot)$, the function 
	\begin{equation*}
		b_v:=\lim\limits_{t \to \infty}b_{t,v}
	\end{equation*}
	is defined everywhere on $M$. The function $b_v$ is called the Busemann function associated to $v\in SM$. The level sets of Busemann functions are called horospheres. Horospheres are closed and hence are complete hypersurfaces with the induced metric. It is classical that for Hadamard manifolds the Busemann functions are $C^2$- convex functions \cite[\S 1.6.6]{Eberlein:1996}. As observed in \cite{Eschenburg:1977}, the curvature restrictions can be relaxed by making use of convergence properties of Jacobi tensors, thereby generalizing the $C^2$- regularity of horospheres to manifolds of continuous asymptote. \\
     
    \noindent In order to obtain this, the following relation between the Hessian of the Busemann function $b_v$ and the stable Jacobi tensor $D_v$ for each $v \in SM$ is used. Given $v \in SM$ and $q \in M$ there exists $w \in S_qM$ asymptotic to $v$ given via the gradient of the Busemann function i.e., $w=-\nabla b_v(q)$ \cite[Proposition 1]{Eschenburg:1977}. We first make a small observation. The orthogonal projection $x^{\perp}(p,v)$ of $x \in T_pM$ onto $(\nabla b_v(p))^{\perp}$ will be denoted by $x^{\perp}$ for notational convenience.
	\begin{lemma}\label{geodesic}
		For $p \in M$ and $x \in S_pM$, we have $\mathrm{Hess}~b_v(p)(x,x)=\mathrm{Hess}~b_v(p)(x^{\perp},x^{\perp})$.
	\end{lemma}
	\begin{proof}
		As $\nabla b_v$ is a $C^1$- vector field with $\norm{\nabla b_v}=1$, we have that $\nabla_{\nabla b_v}\nabla b_v=0$. Indeed $\ip{\nabla_X \nabla b_v}{Y}= \mathrm{Hess}~b_v(X,Y)$ is symmetric in $X,Y$ and hence for every vector field $X$
		\begin{align*}
			\ip{\nabla_{\nabla b_v}\nabla b_v}{X} =\ip{\nabla_X \nabla b_v}{\nabla b_v} =\frac{1}{2}X\ip{\nabla b_v}{\nabla b_v}=0.
		\end{align*}
		In particular, we see that every integral curve of $\nabla b_v$ is a geodesic and hence complete. Then, the orthogonal decomposition $T_pM = \R \nabla b_v(p) \oplus (\nabla b_v(p))^{\perp}$ enables us to write $x=\ip{x}{\nabla b_v(p)}_{T_pM}\nabla b_v(p)+x^{\perp}$. Using $\nabla_{\nabla b_v}\nabla b_v=0$ and the bilinearity of $\mathrm{Hess}~b_v$ we have 
		\begin{equation*}
			\mathrm{Hess}~b_v(p)(x,x)=\mathrm{Hess}~b_v(p)(x^{\perp},x^{\perp}).
		\end{equation*}
	\end{proof}
    
	\noindent To the best of our knowledge, it is not known if the Busemann functions are even convex for an arbitrary manifold of continuous asymptote. Using Lemma \ref{geodesic} and \cite[Proposition 2, Theorem 1]{Eschenburg:1977} we have the following relation between the stable Jacobi tensors and the Busemann functions in the case of manifolds of continuous asymptote as in \cite[Theorem 1]{Eschenburg:1977}. In this case, we have for $p \in M$, $v \in SM$, and $x \in T_pM$ \cite[p. 155]{Knieper:2016},
	\begin{equation}\label{hessian}
		\mathrm{Hess}~{b_{v}}(p)(x,x)  = \ip{D_{-\nabla b_v(p)}'(0)x^{\perp}}{D_{-\nabla b_v(p)}(0)x^{\perp}}.
	\end{equation}
    
    \noindent A manifold is said to have bounded asymptote \cite[Proposition 5]{Eschenburg:1977} if there exists a uniform bound $k \geq 1$ for the stable Jacobi tensors namely,
        \begin{equation}\label{bounded asymptote}
        \norm{D_v(t)} \leq k, \quad \forall ~ v \in SM, ~ t \geq  0.
        \end{equation}
        As observed in \cite{Eschenburg:1977} and stated explicitly in \cite[Remark 10]{KnieperParkPeyerimhoff:2025}, these are necessarily manifolds of continuous asymptote. Note that simply connected manifolds without focal points fall in this class with $k=1$ \cite[\S 5]{Eschenburg:1977}. Except for Section $5$ where we will be in a much broader class, we shall make use of the foregoing discussion for simply connected Riemannian manifolds without focal points and henceforth make this assumption.  \\

   \section{Some applications of convex functions in manifolds without focal points}
    
    \noindent In this section, we consider strictly convex functions on a simply connected manifold $M$ without focal points derived from the distance function. For a fixed $p \in M$, the function $q\mapsto d^2(p, q)$ is smooth and strictly convex where $d(p,q) = : r(q)$ is the distance function. This is proved by an analysis of Jacobi fields as follows. Note that the normalized distance-squared function $u(q)=\frac{1}{2}r^2(q)$ is smooth on $M$. The Hessian of $u$ is given by
   
    \begin{equation*}
        \mathrm{Hess}~u  = \mathrm{d}r^2 + r\mathrm{Hess} ~r.
    \end{equation*}
    Given $v \in T_pM$, consider the Jacobi field along $t \mapsto \exp_p(tv)$ satisfying $Y_v(0)=0$. Then a routine calculation yields,
   
    \begin{equation} \label{hessian of distance function}
        \mathrm{Hess}~r(\exp_p(tv)) (Y_v(t),Y_v(t)) = \ip{Y_{v}'(t)}{Y_v(t)} = \frac{1}{2}\frac{d}{ds}\Big|_{s=t}\norm{Y_v(s)}^2.
    \end{equation}
    
   \noindent Since $M$ has no focal points, $d\exp_p$ is nonsingular at $tv$ and for any $w \in T_qM$ there exists a Jacobi field $Y_v$ along $\gamma_{_v}$ such that $Y_v(0)=0$ and $Y_v(t)=w$. Using \cite[Theorem 4]{Sullivan:1974} we see that $\mathrm{Hess}~ r >0$ for $t > 0$ and so $u$ is strictly convex on $M$ and we refer to \cite{Eberlein:1972} for more details. \\

   \noindent Now we construct, in the usual way, a smooth strictly convex function on $M$ whose gradient is bounded by taking $f(q)=\sqrt{1+2u(q)}$; then $\norm{\nabla f} \leq 1$ and
   \begin{align*}
       \mathrm{Hess}~f & = \frac{1}{(1+2u)^{\frac{3}{2}}}\mathrm{d}r^2 + \sqrt{\frac{2u}{1+2u}} \mathrm{Hess}~r
   \end{align*}
   is positive definite as $\mathrm{Hess}~r$ is strictly positive showing that $f$ is strictly convex. A real-valued function $\varphi$ on a manifold is called an exhaustion function if all its sub-level sets $\{\varphi\leq c\}$ are compact. Notice that the sub-level sets of $f$ are given by
   \begin{equation*}
     f^{-1}(c) =
       \begin{cases}
           \emptyset, & \text{if} ~ c<1 \\
           \{p\}, &  \text{if}~ c = 1 \\
           B(p,c), & \text{if} ~ c > 1
       \end{cases}
   \end{equation*}
   and hence are all compact. Here $B(p, c)$ denotes the closed ball of radius $c$ centred at $p$. Thus $f$ is also an exhaustion function on $M$. For convenience, we summarize this discussion as a lemma.
   \begin{lemma}\label{exhaustion function}
       Let $M$ be a simply connected Riemannian manifold without focal points. Then it possesses a smooth strictly convex function which is Lipschitz and an exhaustion.
   \end{lemma}

    \noindent As a first application we have the following result on the isoperimetric profile function $\mathcal{I}_M$ on $M$.  The isoperimetric profile of a Riemannian manifold is the function that assigns to a given positive volume the infimum of the perimeter of the sets of this volume. We remark here that regularity properties of this function are of interest and refer to the introduction in \cite{Ritore:2017} for a detailed discussion. The continuity and the nondecreasing nature of the isoperimetric profile is shown for complete manifolds which possess smooth strictly convex Lipschitz exhaustion functions \cite[Theorem 3.1]{Ritore:2017}. It then follows from a theorem of Lebesgue that $\mathcal{I}_M$ is differentiable almost everywhere on such manifolds. As examples there, Hadamard manifolds and manifolds with strictly positive sectional curvatures are discussed. We give below a class of examples for which the continuity of $\mathcal I_M$ holds while making no assumptions on curvature and which properly contains Hadamard manifolds. As noted before this class contains plenty of examples whose sectional curvatures do not have a fixed sign \cite{Gulliver:1975}.
     \begin{proposition}\label{isoperimetric profile}
        Simply connected Riemannian manifolds without focal points have a continuous nondecreasing isoperimetric profile function. In particular, they are differentiable almost everywhere. 
    \end{proposition}
     \begin{proof}
    By Lemma \ref{exhaustion function}, $M$ admits a smooth strictly convex Lipschitz exhaustion function.  Invoking  \cite[Theorem 3.1]{Ritore:2017}, we see that $\mathcal{I}_M$ is nondecreasing and continuous.
     \end{proof}
    \noindent Regularity for $\mathcal{I}_M$ has many interesting consequences; for instance, we have the following result whose proof makes use of standard arguments.
    \begin{proposition}
         Let $M$ be as in Proposition \ref{isoperimetric profile}. Suppose $v_0$ is a differentiable point of $\mathcal{I}_M$. If $\Omega$ is an isoperimetric region in $M$ of volume $v_0$, then (the regular part of) its boundary $\partial \Omega$ has constant mean curvature $H = \mathcal{I}_{M}'(v_0)$.
    \end{proposition}
       
    \begin{proof}
    Since $\Omega$ is isoperimetric region the (regular part of) $\partial \Omega$ has constant mean curvature $H$. Let $\mu, \nu$ be the outward and inward normal variational vector fields on $\partial \Omega$ with compact support. Then the first variations of the volumes and the boundary areas are as follows
         \begin{align*}
             \frac{d}{dt}V(\Omega_t)\Big|_{t=0} & = \int\limits_{\partial \Omega_0} |\mu| \mathrm{d}A, \quad \frac{d}{dt}V(\Omega_t)\Big|_{t=0} = - \int\limits_{\partial \Omega_0} |\nu| \mathrm{d}A \\
             \frac{d}{dt}A(\partial \Omega_t)\Big|_{t=0} & = H\int\limits_{\partial \Omega_0} |\mu| \mathrm{d}A, \quad \frac{d}{dt}A(\partial\Omega_t)\Big|_{t=0} = - H\int\limits_{\partial \Omega_0} |\nu| \mathrm{d}A.
         \end{align*}
    Hence, using the differentiability assumption at $v_0$,
        \begin{equation*}
            \lim\limits_{\delta v \to 0^+}\frac{\mathcal{I}_M(v_0+\delta v)-\mathcal{I}_M(v_0)}{\delta v} \leq H, \quad \lim\limits_{\delta v \to 0^-}\frac{\mathcal{I}_M(v_0+\delta v)-\mathcal{I}_M(v_0)}{\delta v}  \geq H
        \end{equation*}
        yields $\mathcal{I}_{M}'(v_0)=H$.
    \end{proof}
    
    \noindent As a consequence, if $v_0$ is a differentiable point of the isoperimetric profile function, then all isoperimetric regions in $M$ of the same volume $v_0$ must also have the same mean curvature. In general, the existence of isoperimetric domains of a given volume in a complete noncompact Riemannian manifold is a difficult problem (see \cite{Ritore:2023}). \\
    
    \noindent We now discuss a couple of direct applications in the situation when $M$ is, in addition, K\"ahler. For general facts on K\"ahler manifolds we refer to \cite[\S7]{GreeneWu:1977}. On a K\"ahler manifold the complex analogue of convex functions is given by the notion of plurisubharmonic functions. These are defined using the complex Hessian called the Levi form. In local coordinates, the Levi form of a function $\tilde{f}$ is given by
      \begin{equation*}
          L\tilde{f} = 4 \sum\limits_{i,j=1}^{n}\frac{\partial^2 h }{\partial z_i \partial \overline{z}_j} \mathrm{d}z_i \otimes \mathrm{d} \overline{z}_j
     \end{equation*}
    where $(z_1,\cdots, z_n)$ is a local holomorphic coordinate system. Let $(x_1,y_1, \cdots, x_n,y_n)$ be the corresponding normal coordinates with $z_k=x_k+iy_k$, $k=1,\cdots,n$. A real-valued function on a K\"ahler manifold is said to be (strictly) plurisubharmonic, if the Levi form is positive (definite) semi-definite. Since $M$ is K\"ahler, the Levi form $L_{\tilde{f}}$ of $\tilde{f}$  is expressed in terms of Hessian of $\tilde{f}$ using the real coordinates as
    \begin{equation*}
       L_{\tilde{f}} = \mathrm{Hess}~\tilde{f} \left(\frac{\partial }{\partial x_1},\frac{\partial }{\partial x_1} \right)+\mathrm{Hess}~\tilde{f} \left(\frac{\partial }{\partial y_1},\frac{\partial }{\partial y_1}\right). 
     \end{equation*}
    Therefore, if $\tilde{f}$ is a strictly convex $C^2$- function on a K\"ahler manifold, then it is strictly plurisubharmonic. \\

    \noindent Stein manifolds are complex manifolds which are biholomorphic to a closed complex submanifold of $\mathbb C^n$ for some $n\in\mathbb N$. This tells us that there are plenty of holomorphic functions on such manifolds. Stein manifolds are necessarily K\"ahler and a question of great interest then is to ask which K\"ahler manifolds are Stein. Typically, sufficient conditions are given by putting constraints on the sign of sectional, holomorphic bisectional or Ricci curvatures \cite[Theorem $1$]{GreeneWu:1977}. It is in this context that we observe the following sufficient condition where we make no assumptions on curvature.

    \begin{proposition}\label{Stein}
    Let $M$ be a simply connected Riemannian manifold without focal points which is also K\"ahler. Then $M$ is necessarily Stein.
    \end{proposition}
    \begin{proof}
    Lemma \ref{exhaustion function} gives a smooth strictly convex exhaustion function on $M$, while the discussion above on the Levi form shows that this function is a strictly plurisubharmonic exhaustion function on $M$. Then Grauert's proof of Levi's problem \cite[\S 4]{GreeneWu:1977} that a complex manifold is Stein if and only if it admits a strictly plurisubharmonic exhaustion function concludes the proof.
    \end{proof}

    \noindent As another application of Lemma \ref{exhaustion function}, we have the following volume growth comparison for metric balls in K\"ahler manifolds without focal points using \cite[Lemma 4.2 (i)]{NiTam:2003}. 
    
    \begin{proposition}\label{volume growth}
    Let $M$ be a simply connected Riemannian manifold without focal points which is also K\"ahler and of (complex) dimension $m$. Then $M$ admits a  smooth strictly plurisubharmonic function with bounded gradient. Moreover, for each $p \in M$ there exists a constant $C_2 >0$ such that, 
        \begin{equation*}
               \mathrm{Vol}(B_p(r)) \geq C_2 r^m, \quad r \geq 2
         \end{equation*}
    where $\mathrm{Vol}(B_p(r))$ is the volume of the metric ball $B_p(r)$ of radius $r$ about $p$ in $M$.
    \end{proposition}
    
    \noindent We remark here that the question of whether a simply connected harmonic manifold that carries a Kähler structure must be a symmetric space is still unresolved. However, in the context of Damek–Ricci spaces this has been verified  \cite[\S 4.1.12]{Vanhecke:1995}.\\ 
   
    \noindent  While the theme of this article is to use growth conditions of Jacobi tensors instead of bounds on sectional curvature,  imposing lower bounds on sectional curvatures in addition to growth conditions of Jacobi tensors gives better estimates for volume growth (including in the K\"ahler case). Let $p\in M$ be arbitrary but fixed. Let $A_v$ be the Lagrange tensor along $\gamma_{_v}$ satisfying $A_v(0) = 0$ and $A_{v}'(0)=\mathrm{Id}$ for each $v \in S_pM$. Then the area of a geodesic sphere $S_p(t)$ of radius $t$ about $p$ is given by \cite[p. $151$]{Knieper:2016}
    \begin{equation*}
          \mathrm{Area}(S_p(t))=\int\limits_{S_pM} \det(A_v(t)) ~ \mathrm{d}\theta_p(v).
    \end{equation*}
    If the sectional curvatures of a manifold $M$ with bounded asymptote are bounded below, say by $-\beta^2$ for $\beta >0$, then using \cite[Proposition 6]{Eschenburg:1977}, we obtain 
    \begin{equation*}
        \det(A_v(t)) \geq \left(\frac{1}{2k\sqrt{\beta}}\sqrt{t}\right)^{n-1} , \quad \forall ~ v \in SM, ~  t \geq T= \frac{1}{\beta} \coth^{-1} \left(\frac{2}{\beta} \right).
    \end{equation*}
    where $k$ is as in \eqref{bounded asymptote}. We then get the following growth rate of the volume of geodesic balls.
       \begin{align*}
           \mathrm{Vol}(B_p(r)) & = \int\limits_{0}^{r} \int\limits_{S_pM} \det(A_v(t)) ~\mathrm{d}\theta_p(v) \mathrm{d}t \\
           & \geq \int\limits_{T}^{r}\left(\frac{1}{2k\sqrt{\beta}}\sqrt{t}\right)^{n-1} \left(\int\limits_{S_pM}  ~\mathrm{d}\theta_p(v) \right)\mathrm{d}t \\
          & = \omega_p\int\limits_{T}^{r} \left(\frac{1}{2k\sqrt{\beta}}\sqrt{t}\right)^{n-1} dt, \quad  \\
           & = \omega_p\frac{2}{n+1}\left(\frac{1}{2k\sqrt{\beta}}\right)^{n-1} \left(r^{\frac{n+1}{2}} - T^{\frac{n+1}{2}}  \right), \quad r \geq T.
       \end{align*}
    where $\omega_p = \displaystyle{\int\limits_{S_pM}  ~\mathrm{d}\theta_p(v)}$ is the area of $S_pM$ with respect to the Riemannian measure on $S_pM$.

    \subsection{Absolutely continuous spectrum via radial curvatures} Our next application is spectral theoretic in nature, and so we begin with a sketch of the basic functional analytic setup concerning the spectrum of the Laplacian referring to \cite{Borthwick:2020} for the details. Let $M$ be a complete Riemannian manifold. Then the (positive) Laplacian $\Delta$ is an essentially self-adjoint operator with the space of compactly supported smooth functions on $M$ as the domain. We also denote its unique self-adjoint extension, having the domain $\mathcal{D}(\Delta):=\{f \in L^2(M) : \Delta f \in L^2(M)\}$, by $\Delta$. Using the functional calculus for $\Delta$ we have the family of spectral projection operators $\prod_E:=\chi_E(\Delta)$ for each Borel subset $E$ of $\R$ referred to as the resolution of $\Delta$.  The support of these projections and the spectrum of $\Delta$ are related in the following way. We say that $\lambda \in \C$ is in the spectrum $\sigma(\Delta)$ of $\Delta$ if and only if $\prod_{(\lambda -\epsilon,\lambda+\epsilon)} \neq 0$ for all $\epsilon >0$. If $\prod_{\{\lambda\}} \neq 0$, then $\lambda$ is an eigenvalue of $\Delta$ and the range of $\prod_{\{\lambda\}}$ is the corresponding eigenspace. The spectral theorem gives a unitary equivalence between $\Delta$ and a multiplication operator on a suitable Hilbert space of the form $\mathrm{L}^2(X,\mu)$, where $\mu$ is a $\sigma$- finite Borel measure on a measure space $X$. By the Lebesgue decomposition, $\mu$ decomposes as $\mu= \mu_{pp}+\mu_{ac}+\mu_{sc}$ into pure point measure, absolutely continuous and singularly continuous measures with respect to the Lebesgue measure, and consequently we have that
	\begin{equation}
		\mathrm{L}^2(X,\mu)=\mathrm{L}^2(X,\mu_{pp}) \oplus \mathrm{L}^2(X,\mu_{ac})\oplus \mathrm{L}^2(X,\mu_{sc}).
	\end{equation}
	This induces a decomposition $\mathrm{L}^2(M) = \mathcal{H}_{pp} \oplus \mathcal{H}_{ac} \oplus \mathcal{H}_{sc}$ from which we have the associated decomposition of the spectrum of $\Delta$ namely, $\sigma(\Delta)=\overline{\sigma_{pt}(\Delta)} \cup \sigma_{ac}(\Delta) \cup \sigma_{sc}(\Delta)$ where 
	\begin{equation*}
		\sigma_{ac}(\Delta)=\sigma(\Delta_{|_{\mathcal{H}_{ac}}}), \quad  \sigma_{sc}(\Delta)=\sigma(\Delta_{|_{\mathcal{H}_{sc}}}), \quad  \overline{\sigma_{pt}(\Delta)}=\sigma(\Delta_{|_{\mathcal{H}_{pp}}}).
	\end{equation*}
	Here we take the closure for the set of all eigenvalues as it is not necessarily closed. We refer to $\sigma_{ac}(\Delta)$ and $\sigma_{sc}(\Delta)$ as the absolutely continuous and singularly continuous parts of $\sigma(\Delta)$ respectively and their union as the continuous part of $\sigma(\Delta)$. There is also a decomposition of $\sigma(\Delta)$ as a disjoint union: 
	\begin{equation*}
		\sigma(\Delta)= \sigma_{disc}(\Delta) \cup \sigma_{ess}(\Delta)
	\end{equation*}
    	where $\sigma_{disc}(\Delta)$ and $\sigma_{ess}(\Delta)$ denote the discrete and essential parts of $\sigma(\Delta)$. The former consists of all eigenvalues of finite multiplicities and hence is contained in $\sigma_{pt}(\Delta)$. The latter is the complement of $\sigma_{disc}(\Delta)$ in $\sigma(\Delta)$ thereby  we have that the continuous part of the spectrum, namely $\sigma_{ac}(\Delta) \cup \sigma_{sc}(\Delta)$ is contained in $\sigma_{ess}(\Delta)$. Note that it is possible that $\sigma_{ess}(\Delta)$ contains eigenvalues of infinite multiplicities usually referred to as the embedded eigenvalues. \\
    
    \noindent In order to analyse the nature of the spectrum of the Laplacian using convexity, we make use of \cite[Theorem 3]{Xavier:1988}. The idea there was to use Kato's $H$- smooth operators to show that if $M$ admits a strictly convex function, then certain boundedness conditions on its gradient and Laplacian implies the nontriviality of the absolutely continuous part of the spectrum. This is then applied in the context of Hadamard manifolds with certain additional assumptions which are satisfied with lower bounds on sectional curvatures, to show the nontriviality of the absolutely continuous spectrum. We show that this result holds true for simply connected Riemannian manifolds without focal points having nonpositive radial curvatures with respect to a single point. 
    \noindent To this end, we prove the following lemma.

    \begin{lemma}\label{hessian bounds}
    Let $M$ be a simply connected Riemannian manifold without focal points. Suppose there exists a point $p$ in $M$ such that the radial curvatures with respect to $p$ are nonpositive. Then we have for any $q \neq p$ in $M$
    \begin{equation*}
          r(q) \mathrm{Hess}~r(q) (v,v) \geq \norm{v}^2, \quad \forall ~ v \in T_qM.
    \end{equation*}
    \end{lemma}
    \begin{proof}
     Since $M$ has no focal points $\exp_p:T_pM \to M$ is a diffeomorphism. Hence for any $q \in M$, $q\neq p$, there exists a unique $t \in \R$ and a unit tangent vector $v \in T_pM$ such that $q = \exp_p(tv)$. Moreover, $\gamma_{_v}(s)=\exp_p(sv)$ is a unit speed geodesic in $M$ joining $p$ and $q$ satisfying $\gamma_{_v}(0)=p$ and $\gamma_{_v}(t)=q$. Then $\gamma_{_v}'(s)=\nabla r(\gamma_{_v}(s))$ and in particular, $\gamma_{_v}'(t)=\nabla r(q)$. Since $d\exp_p$ is nonsingular at $tv$, for any $w \in T_qM$ there exists a Jacobi field $Y_{v}$ along $\gamma_{_v}$ such that $Y_{v}(0)=0$ and $Y_{v}(t)=w$. Notice that by our assumption on nonpositive radial curvature, we have $\ip{R(Y_{v}(t),\gamma_{_v}'(t))\gamma_{_v}'(t)}{Y_{v}(t)} \leq 0$ for any $t >0$ and $v \in S_pM$. We now recall the standard technique of Jacobi field comparison. Since the second derivative of the Jacobi field is related to the curvature tensors via the Jacobi equation, quantitative estimates are obtained by imposing conditions on curvatures. Note that
    \begin{equation*}
        \lambda(s)=\frac{\norm{Y_{v}(s)}^2}{\ip{Y_{v}'(s)}{Y_{v}(s)}}, ~ s > 0
    \end{equation*}
    is well defined since $M$ has no focal points. Using l'H\^{o}pital's rule, we see that $\lambda(0)=0$ since $Y_{v}'(0) \neq 0$. Differentiating and using the Jacobi equation we have
    \begin{align*}
        \lambda'(s) & =\frac{2\ip{Y_{v}'(s)}{Y_{v}(s)}^2-\left(\norm{Y_{v}'(s)}^2 - \ip{R(\gamma_{_v}'(s),Y_{v}(s))\gamma_{_v}'(s)}{Y_{v}(s)}\right)\norm{Y_{v}(s)}^2}{\ip{Y_{v}'(s)}{Y_{v}(s)}^2} \\
        & \leq \frac{2\ip{Y_{v}'(s)}{Y_{v}(s)}^2-\norm{Y_{v}'(s)}^2\norm{Y_{v}(s)}^2 }{\ip{Y_{v}'(s)}{Y_{v}(s)}^2} \quad\text{(nonpositivity of radial curvature)} \\
        & \leq \frac{2\ip{Y_{v}'(s)}{Y_{v}(s)}^2-\ip{Y_{v}'(s)}{Y_{v}(s)}^2}{\ip{Y_{v}'(s)}{Y_{v}(s)}^2} = 1 \quad \text{(Cauchy-Schwarz inequality)}.
    \end{align*}
    Hence $\lambda(s) \leq s$ or equivalently, 
    \begin{equation*}
        s \ip{Y_{v}'(s)}{Y_{v}(s)} \geq \norm{Y_{v}(s)}^2, \quad s > 0.
    \end{equation*}
    The claim follows from \eqref{hessian of distance function}.
    \end{proof}

    \noindent We are now ready to prove the main spectral result of this section. The bounds on the Laplacian and the bi-Laplacian demanded below are satisfied if lower bounds on curvatures and derivatives of metric tensor are assumed. It should be kept in mind that there exist Hadamard manifolds with pure point spectrum if the curvatures become unbounded. 

    \begin{remark}
    Note that we adopt the convention of considering the Laplacian as a positive operator, in contrast to the convention used in \cite{Xavier:1988} for Hadamard manifolds. Consequently, certain inequalities in the subsequent theorem appear with opposite signs.
    \end{remark}

    \begin{theorem}\label{radial absolutely continuous}
    Let $M$ be a simply connected Riemannian manifold without focal points. Suppose there exists a point $p$ in $M$ such that the radial curvatures with respect to $p$ are nonpositive. If there exist constants $c_1,c_2>0$ such that $f$ and the Laplacian $\Delta$ satisfies 
     \begin{enumerate}
         \item [(i)] $ \Delta f \geq - c_1 $
         \item[(ii)] $\Delta^2 f \leq c_2 f^{-3}$
     \end{enumerate}
     then the absolutely continuous spectrum of $\Delta$ is nontrivial. In fact, $(\alpha,\infty)$ is in the absolutely continuous part of the spectrum of $\Delta$ where $\alpha = \frac{3}{2}c_1+\frac{1}{4}c_2$.
    \end{theorem}
     
    \begin{proof}
    Note that $\norm{\nabla f}$ is bounded. Since $M$ has nonpositive radial curvature at $p$, from Lemma \ref{hessian bounds} we obtain that $\mathrm{Hess}~r^2(p)$ is the Riemannian metric at $p$ and for any $p \neq q \in M$,
    \begin{equation*}
    \mathrm{Hess}~r^2(q)(w,w)  = \ip{\nabla r(q)}{w}^2 + r(q) \mathrm{Hess}~r (q)(w,w) \geq \norm{w}^2, \quad \forall ~ w \in T_q M, \norm{w}=1.
    \end{equation*}
    From here, the arguments are standard, as in \cite[Theorem 1]{Xavier:1988}.  We write them down here for completeness. Recall that $f(q) = (1+r(q)^2)^{1/2}$ and its Hessian can be computed to be
    \begin{align*}
    \mathrm{Hess}~f(q)(w,w) & = -\frac{r^2(q)}{(1+r^2(q))^{3/2}}\ip{w}{\nabla r(q)}^2 + \frac{1}{2(1+r^2(q))^{1/2}}\mathrm{Hess}~r^2(q)(w,w) \\
    & \geq \frac{1}{(1+r^2(q))^{3/2}}.
    \end{align*}
    Taking $h(q)=\frac{1}{(1+r^2(q))^{3/2}}$, we see that 
    \begin{align*}
    \Delta h & =  \frac{3}{f^4}\Delta f + \frac{12}{f^{5}} \norm{\nabla f}^2, \\
    \implies \Delta h & \geq \frac{3\Delta f} {f^4} \geq - \frac{3c_1}{2}.
    \end{align*}
    Thus for
    \begin{equation*}
    \beta = \inf \frac{1}{4h}\{2\Delta h - \Delta^2 f \} \geq - \frac{3c_1}{2}-\frac{c_2}{4} 
    \end{equation*}
    the claim now follows as in \cite[Theorem 3]{Xavier:1988}.
    \end{proof}
    \noindent Note that that the geometric constraints in the form of \textit{no focal points} and \textit{nonpositive radial curvature at a point} that we impose on the manifold are apriori weaker than being a Hadamard manifold and the above theorem generalizes \cite[Theorem 1]{Xavier:1988}.  However, more refined questions such as identifying whether the point spectrum or the singularly continuous spectrum occur are not amenable to this approach via the use of the distance function. To overcome this, we now turn our attention on the other natural geometrically defined convex function, the Busemann function.

    \section{A construction using Busemann functions and its applications to spectrum}  In this section, we construct a strictly convex function out of Busemann functions on simply connected manifolds without focal points, using the geometry at infinity and assuming, in addition, strict negativity of Ricci curvature. As remarked in the introduction, this is motivated by a similar construction on symmetric spaces of noncompact type outlined in  \cite[penultimate paragraph p. $657$]{DonnellyXavier:2006}. The main purpose of this construction here is that it allows us to prove new results (Theorem \ref{spectrum of asymptotically harmonic manifolds} and Corollary \ref{harmabs}) completely determining the nature of the spectrum for certain interesting classes of manifolds without focal points whose horospheres have constant mean curvature at a point. \\
    
    \noindent We describe briefly the geodesic compactification of $M$ and its connections to the Riemannian exponential map and the Busemann functions. We refer to \cite{Goto:1979} for more details on what follows. Given $v \in SM$ there exists a unique geodesic $\gamma_{_v}$ in $M$ satisfying $\gamma_{_v}'(0)=v$. Two vectors $v,w \in SM$ are said to be asymptotic if the distance between $\gamma_{_v}(t)$ and $\gamma_{_w}(t)$ is bounded for all $t \geq 0$. Let $M(\infty)$ be the set of all classes of asymptotic vectors and let $\overline{M}=M \cup M(\infty)$. $\overline{M}$ has a canonical topology with the following property: for any $p \in M$, the exponential map $\exp_p:T_pM \to M$ extends uniquely to a homeomorphism $\Phi_p:\overline{T_pM} \to \overline{M}$ where $\overline{T_pM} \backslash T_pM$ is homeomorphic with the unit sphere $S_p M$ in $T_pM$. Then for any $p \in M$, there exists a unique geodesic $\gamma_{_{-\nabla b_{_v}(p)}}$ starting at $p$ asymptotic to $\gamma_{_v}$. From our discussions in \S 2, we see that the Busemann boundary $\partial^o M$ and the geodesic compactification $M(\infty)$ are homeomorphic. Further, for each $v \in SM$, the Busemann function $b_v$ is convex \cite[Theorem 2]{Eschenburg:1977} i.e., $\mathrm{Hess}~b_v$ is positive semidefinite for every $v \in SM$. It is, however, not strictly convex as can be seen from Lemma \ref{geodesic}. The convexity of the Busemann functions imply that each horosphere is a convex submanifold in $M$. We are now ready to prove the following theorem.

	\begin{theorem}\label{strictly convex function}
		Let $M$ be a simply connected Riemannian manifold with no focal points and negative Ricci curvature. Then $F(p)=\displaystyle{\int\limits_{M(\infty)}b_v(p)~\der{d}\eta(v)}$ is a strictly convex function whose gradient is bounded.
	\end{theorem}
	\begin{proof}
	Let $o$ be an arbitrary but fixed point in $M$. Let $\{v_k\}$ be a countable dense subset in $S_oM$. Consider the Dirac measures $\delta_{v_k}$ associated to each $v_k$. Then $\displaystyle{\eta=\sum\limits_{k=1}^{\infty}\frac{1}{2^k}\delta_{v_k}}$ defines a  probability measure on $S_oM$.  We identify $S_o M$, $M(\infty)$ and $\partial^o M$ via the respective homeomorphisms.
    By dominated convergence theorem, the gradient of $F$ at $p$ is given by $\nabla F(p)=\displaystyle{\int\limits_{M(\infty)}\nabla b_v(p)~\der{d}\eta(v)}$. Clearly,
		\begin{equation*}
			\norm{\nabla F(p)} = \norm{~~\int\limits_{M(\infty)}\nabla b_v(p)~\der{d}\eta(v)}\leq \int\limits_{M(\infty)}\norm{\nabla b_v(p)}~\der{d}\eta(v)= 1 ,\quad \text{as} ~\norm{\nabla b_v}= 1
		\end{equation*}
		i.e., $F$ has bounded gradient. In particular, $F$ is Lipschitz with Lipschitz constant $1$. For $p \in M$ and $w \in S_pM$, using \eqref{hessian} and that $\mathrm{Hess}~b_v$ is positive semidefinite we have 
		\begin{align*}
			\int\limits_{M(\infty)}\mathrm{Hess}~b_v(p)(w,w) ~\der{d}\eta(v) & =  \left|~~\int\limits_{M(\infty)}\ip{D_{-\nabla b_v(p)}'(0)w}{w} ~\der{d}\eta(v) \right| \\
			& \leq  \int\limits_{M(\infty)}|\ip{D_{-\nabla b_v(p)}'(0)w}{w} |~\der{d}\eta(v) \\
			& \leq \int\limits_{M(\infty)}\norm{D_{-\nabla b_v(p)}'(0)w} \norm{w} ~\der{d}\eta(v) \quad \text{(using Cauchy-Schwarz)}\\
			& \leq \int\limits_{M(\infty)}\norm{D_{-\nabla b_v(p)}'(0)} ~\der{d}\eta(v)  
		\end{align*}
    since $D_{-\nabla b_v(p)}'(0)$ is a bounded linear map for each $v \in S_o M$. The restriction of the continuous map (\S 2 above) $u \mapsto \norm{D_{u}'(0)}$ on $SM$ to $S_pM$ shows that there exists some $L>0$ such that $\norm{D_{u}'(0)} \leq L$ for all $u \in S_pM$. In particular, $ \norm{D_{-\nabla b_v(p)}'(0)} \leq L$ for all $v \in S_oM$. So another application of the dominated convergence theorem shows that 
		\begin{equation}\label{hessian of convex function}
			\mathrm{Hess}~ F(p)(w,w)=\int\limits_{M(\infty)}\mathrm{Hess}~b_v(p)(w,w) ~\der{d}\eta(v).
		\end{equation}
	The convexity of $F$ follows from the convexity of the Busemann functions. We now claim that $F$ is, in fact, strictly convex. Suppose there exists $p\in M$ and $0 \neq w \in T_pM$ such that $\mathrm{Hess}~ F(p)(w,w)=0$ then 
		\begin{equation*}
			\int\limits_{M(\infty)}\mathrm{Hess}~b_v(p)(w,w) ~\der{d}\eta(v) = 0. 
		\end{equation*}
	Since $v \mapsto \mathrm{Hess}~b_v(p)(w,w)$ is a nonnegative continuous function in $v$, we have $\mathrm{Hess}~b_v(p)(w,w)=0$ for all $v \in S_oM$. Using \eqref{hessian}, 
    \begin{equation}\label{critical point}
        \ip{D_{-\nabla b_v(p)}'(0)w^{\perp}}{D_{-\nabla b_v(p)}(0)w^{\perp}}=0, \quad \text{for all}~ v \in S_o M.
    \end{equation}
		
	\noindent It is known from \cite[\S 5]{Eschenburg:1977} that for $v \in SM$ and all parallel vector fields $x$ normal to $\gamma_{_v}$, $\frac{d}{dt}(\norm{D_v(t)x}^2) \leq 0$ and so the function $f_{v,x}(t)=\frac{1}{2}\norm{D_v(t)x}^2$ is $C^2$- and monotonically decreasing. From \eqref{critical point} we see that for each $v \in S_oM$, $f_{-\nabla b_v(p),w^{\perp}}$ has a critical point at $0$. We claim that $f_{-\nabla b_v(p),w^{\perp}}''(0)= 0$. If not, then $0$ is either a local minimum or local maximum. Suppose $0$ is a local minimum then there exists a $\delta$-neighbourhood of $0$ such that $f_{-\nabla b_v(p),w^{\perp}}(0) \leq f_{-\nabla b_v(p),w^{\perp}}(t)$ for all $t \in (-\delta,\delta)$. Since $f$ is monotonically decreasing we have $f_{-\nabla b_v(p),w^{\perp}}(t)=f_{-\nabla b_v(p),w^{\perp}}(0)$ for all $t \in (0,\delta)$. Hence $f_{-\nabla b_v(p),w^{\perp}}$ is constant on $(0,\delta)$. Observe that $f_{-\nabla b_v(p),w^{\perp}}$ is $C^2$- and so $f_{-\nabla b_v(p),w^{\perp}}''(0)=\lim\limits_{t \to 0_{+}}f_{-\nabla b_v(p),w^{\perp}}''(t)=0$, a contradiction. A similar argument could be carried out to show that $0$ cannot be a local maximum of $f_{-\nabla b_v(p),w^{\perp}}$. Hence $0$ is a degenerate critical point of $f_{-\nabla b_v(p),w^{\perp}}$ i.e, $f_{-\nabla b_v(p),w^{\perp}}''(0)=0$. From the Jacobi equation for $-D_{-\nabla b_v(p)}$ we have
		\begin{equation*}
			D_{-\nabla b_v(p)}''(0)+R_{-\nabla b_v(p)}(0) D_{-\nabla b_v(p)}(0)= 0
		\end{equation*}
	and this gives
        \vspace*{-3mm}
		\begin{align}\label{sectional}
			0 & = f_{-\nabla b_v(p),w^{\perp}}''(0) \nonumber \\
			& = \ip{D_{-\nabla b_v(p)}''(0)w^{\perp}}{D_{-\nabla b_v(p)}(0)w^{\perp}} +\norm{D_{-\nabla b_v(p)}'(0)w^{\perp}}^2 \nonumber \\
			& = -\ip{ {R_{-\nabla b_v(p)}}(0)w^{\perp}}{w^{\perp}} + \norm{D_{-\nabla b_v(p)}'(0)w^{\perp}}^2 \nonumber \\
			\implies \ip{ R_{-\nabla b_v(p)}(0)w^{\perp}}{w^{\perp}}
			& = \norm{D_{-\nabla b_v(p)}'(0)w^{\perp}}^2 \nonumber \\
			\implies K(-\nabla b_v(p),w^{\perp}) & =\norm{D_{-\nabla b_v(p)}'(0)w^{\perp}}^2 \geq 0  , \quad \text{for all $v \in S_oM$}
		\end{align}
	where $K$ denotes the sectional curvature corresponding to the plane spanned by $\nabla b_v(p)$ and $w^{\perp}$. Let $x \in S_pM$ be such that $x \perp w$. Since $S_pM$ is homeomorphic to $M(\infty)$ there exists a unique $v_o \in S_oM$ such that $\gamma_{v_o}$ and $\gamma_x$ are asymptotic. As $\{v_k\}$ is dense in $S_oM$ there exists a subsequence $(v_{_{k(m)}})$ in $S_oM$ converging to $v_o$. For each $m$ we have $-\nabla b_{v_{_{k(m)}}}(p) \in S_pM$ such that $\gamma_{_{-\nabla b_{v_{_{k(m)}}}}(p)}$ and $\gamma_{_{v_{_{k(m)}}}}$ are asymptotic and so $v_{_{k(m)}}$ and $-\nabla b_{v_{_{k(m)}}}(p)$ correspond to the same element in $M(\infty)$. It then follows that the sequence $(-\nabla b_{v_{_{k(m)}}}(p))$ converges to $x$ in $S_pM$. We denote the orthogonal projection of $w$ onto $(-\nabla b_{v_{_{k(m)}}}(p))^{\perp}$ by $w^{\perp}(m)$. Then 
		\begin{align*}
			\norm{w - w^{\perp}(m)}^2 & =
			\norm{\ip{w}{-\nabla b_{v_{_{k(m)}}}(p)}\nabla b_{v_{_{k(m)}}}(p)}^2  \\
			& = |\ip{w}{-\nabla b_{v_{_{k(m)}}}(p)}|^2 \norm{\nabla b_{v_{_k(m)}}(p)}^2 \\
            & = |\ip{w}{-\nabla b_{v_{_{k(m)}}}(p)}|^2  \xrightarrow{m \to \infty} 0 \quad \text{as ~ $-\nabla b_{v_{_{k(m)}}}(p) \to x$ and $x \perp w$}. 
		\end{align*}
	Thus the sequence $(w^{\perp}(m))$ converges to $w$ as $m \to \infty$ and from \eqref{sectional} we have,
		\begin{equation*}
			K(-\nabla b_{v_{_{k(m)}}}(p),w^{\perp}(m)) \xrightarrow{m \to \infty} K(x,w) \geq 0.
		\end{equation*}
		Now for $\{x_1, \cdots, x_{n-1}\}$ orthonormal vectors  in $T_pM$ such that $x_i \perp w$, 
        \begin{equation*}
            Ricc(w,w)=\sum\limits_{i=1}^{n-1}K(x_i,w) \geq 0. 
        \end{equation*}
        This is a contradiction to the hypothesis that $M$ has negative Ricci curvature. Hence $\mathrm{Hess}~F(p)(w,w)>0$ for every $p\in M$ and every $0 \neq w \in T_pM$ proving the strict convexity of $F$.
	\end{proof}
    
    \noindent We examine the hypotheses of Theorem \ref{strictly convex function}. In this class of manifolds, having nonnegative Ricci curvature at even one point, forces the manifold to be flat \cite[Corollary 1]{Zhou:2000}. Being an average of sectional curvatures, the sign of the Ricci curvature is a milder hypothesis.  That this is true especially for negative Ricci curvature can be seen from the remarkable  result \cite[Theorem A]{Lohkamp:1994} that any smooth manifold of dimension at least three admits a complete Riemannian metric with negative Ricci curvature.\\

    \noindent Let us now briefly recall some facts regarding the bottom $\lambda_0(M)$ of the spectrum for a complete noncompact Riemannian manifold $M$. It is characterized by \cite[\S 4.4]{Chavel:1984}
    \begin{equation}\label{Rayleigh quotient}
        \lambda_0(M)=\inf\limits_{f \neq 0} \frac{\int\limits_M |\nabla f|^2}{\int\limits_M |f|^2}
    \end{equation}
    where $f$ varies over the completion $\mathfrak{H}(M)$ of $C_{c}^{\infty}(M)$ with respect to the norm $\norm{f}=\norm{f}_{\rm{L}^2(M)}+\norm{\nabla f}_{\rm{L}^2(M)}$. The Cheeger's constant of $M$ is given by
	\begin{equation}\label{Cheeger constant}
		h_{Cheeger}(M):=\inf\limits_{\Omega \subset M} \frac{\mathrm{Area}(\partial\Omega)}{\mathrm{Vol}(\Omega)}
	\end{equation}
	where the infimum varies over all open, connected submanifolds $\Omega$ of $M$ with compact closure and smooth boundary $\partial \Omega$. The importance of this global geometric invariant is illustrated by the following fundamental inequality (Cheeger's inequality) relating it to the global analytic invariant $\lambda_0(M)$ \cite[\S 4.3, Theorem 3]{Chavel:1984},
	\begin{equation}\label{cheeger inequality}
		\frac{h_{Cheeger}^{2}(M)}{4} \leq \lambda_0(M).
	\end{equation}
    In particular, $\sigma(\Delta_M)$ is contained in $[h_{Cheeger}^{2}/4,\infty)$. \\
    
    \noindent Let $M$ be a manifold of continuous asymptote and let $o$ be an arbitrary but fixed point in $M$. Now, $\{-D_{s,v}'(0)\}_{s>0}$ is strictly monotonically decreasing in $s$ for each $v \in SM$ (\cite[proof of Proposition 2]{Eschenburg:1977} and $v \mapsto D_{v}'(0) $ is continuous. So the compactness of $S_oM$ yields, by Dini's theorem, that the convergence $D_{s,v}'(0) \to D_{v}'(0)$ is uniform. It then follows from \cite[Theorem 1]{Eschenburg:1977} that all the Busemann functions $b_v$ for $v\in S_oM$, and the associated horospheres are $C^2$. Further, we assume that for each $v \in S_oM$, the horospheres associated with $b_v$ all have the same constant mean curvature $H(v)$ i.e., it is independent of points and depends only on the direction $v$.  Observe that this assumption is satisfied, for instance, for all $v \in SM$ when $M$ is a symmetric space of noncompact type (as will be seen later) as well as for asymptotically harmonic manifolds. Using \eqref{hessian}, we see that
    \begin{equation*}
       \Delta b_v(p) = \mathrm{trace}(-D_{-\nabla b_v(p)}'(0))=\mathrm{trace}(-D_{v}'(0)) = H(v), \quad \forall ~ p \in M. 
    \end{equation*}
     Moroever, $H(v)$ depends continuously on $v \in S_oM$ and therefore
    \begin{equation*}
         H:=\sup\limits_{v \in S_oM} H(v)
    \end{equation*}
    is finite. 
   
    \noindent For an open connected $\Omega$ in $M$ with compact closure and smooth boundary $\partial \Omega$ we have, by Stokes' theorem, that
    \begin{equation*}
        H(v) \mathrm{Vol}(\Omega)=  ~\int\limits_{\Omega} \Delta b_v = ~~ \int\limits_{\partial \Omega} \ip{\nabla b_v}{\nu} \leq \left|~~ \int\limits_{\partial \Omega} \ip{\nabla b_v}{\nu} \right|\leq \mathrm{Area}(\partial \Omega), \quad (\text{since $\norm{\nabla b_v}=1$})
    \end{equation*}
    where $\nu$ is the outer unit normal vector field of $\Omega$ along $\partial \Omega$. Thus, 
    \begin{equation*} \label{cheeger is atleast h}
        H(v) \leq \inf\limits_{\Omega \subset M} \frac{\mathrm{Area}(\partial\Omega)}{\mathrm{Vol}(\Omega)} = h_{Cheeger}(M)
    \end{equation*}
    where the infimum varies over $\Omega$ as above in $M$. Taking supremum over all such $v \in S_oM$ we have
    \begin{equation}\label{supremum of mean curvatures}
        \sup\limits_{v \in S_oM}H(v)= H \leq h_{Cheeger}(M).
    \end{equation}
    
    \noindent We remark that the Ricci curvatures of $M$ are nonpositive in the radial directions with respect to $o$. This can be seen as follows. Let $(\phi^t)_{t \in \R}: SM \to SM$ be the geodesic flow on $M$ given by $\phi^t(v)=\gamma_{_v}'(t)=-\nabla b_v(\gamma_{_v}(t))$ with $\gamma_{_v}'(0)=v$.  Note that $-D_{-\phi^{t}(v)}'(0)$ denotes the shape operators associated to the horosphere at $\gamma_{_v}(t)$ with mean curvature $H(v)$. Hence by taking traces in the Riccati equation
     \begin{equation*}
             U'(t) = U(t)^2 + R_{v}(t).
    \end{equation*}
    we have that
    \begin{equation*}
        0 = \mathrm{trace}((-D_{-\phi^{t}(v)}'(0))^2) + \mathrm{Ricc}(\gamma_{_v}'(t),\gamma_{_v}'(t)).
    \end{equation*}\\
   
    \noindent The following theorem shows that for $M$ as in Theorem \ref{strictly convex function} whose horospheres about a point have constant mean curvature, we can completely determine the nature of spectrum.
   
    \begin{theorem}\label{spectrum of asymptotically harmonic manifolds}
    Let $M$ and $F$ be as in Theorem \ref{strictly convex function} and suppose that there exists an $o \in M$ such that the horospheres of the Busemann function $b_v$ for each $v \in S_oM$ all have the same constant mean curvature denoted by $H(v)$. Then
        \begin{itemize}
            \item [(a)]  $F$ has bounded Laplacian and $\Delta^2 F =0$.
            \item[(b)] The spectrum of the Laplacian is purely absolutely continuous and is contained in $[H^2/4,\infty)$ where $H:=\sup\limits_{v \in S_oM}H(v)$. In particular, there are no $\mathrm{L}^2$-eigenvalues.
        \end{itemize}
    \end{theorem}
	\begin{proof}  
        $(a)$ Note that $F$ is convex and $\Delta  F = trace(\mathrm{Hess}~F)$. So by choosing an orthonormal basis $\{E_i(p)\}_{i=1}^{n}$ of $T_pM$ we have
        \begin{align*}
		0 \leq \sum\limits_{i=1}^{n} \mathrm{Hess}~F(p)(E_i(p),E_i(p)) & =  \sum\limits_{i=1}^{n} \int\limits_{M(\infty)}\mathrm{Hess}~b_v(p)(E_i(p),E_i(p)) ~\der{d}\eta(v) \\
		& = \int\limits_{M(\infty)} \sum\limits_{i=1}^{n} \mathrm{Hess}~b_v(p)(E_i(p),E_i(p)) ~\der{d}\eta(v). 
        \end{align*}
		By taking trace
        \begin{equation}\label{constant of laplacian}
             \Delta F(p)  = \int\limits_{M(\infty)} \Delta b_v(p)~\der{d}\eta(v) =  \int\limits_{M(\infty)}H(v) \der{d}\eta(v)
        \end{equation}
        and $(a)$ follows from \eqref{constant of laplacian} since $\Delta F$ is constant, for $H(v)$ is independent of the point $p$ by assumption.\\
        
    $(b)$ From Theorem \ref{strictly convex function} and $(a)$ we have,
    \begin{enumerate}
		\item [(i)] $\norm{\nabla F} \leq 1$, since $\norm{\nabla b_v}=1$ for all $v \in S_oM$.
		\item[(ii)] from \eqref{constant of laplacian}, $\Delta F$ is constant thereby bounded and that $\Delta^2 F=0$.
	\end{enumerate}
    Thus $F$ is a strictly convex function on $M$ satisfying the  conditions (i)-(ii). This enables us to apply Kato's theory of $H$- smooth operators as in \cite[Corollary to Theorem 2]{Xavier:1988} and thence conclude that the spectrum is purely absolutely continuous. It follows from \eqref{cheeger inequality} and \eqref{supremum of mean curvatures} that the spectrum is contained in $[H^2 / 4,\infty)$. 
	\end{proof}

    \begin{remark}
        In fact, we shall see below in Theorem \ref{cheeger and bottom} that $H^2/4$ is the bottom of the spectrum. 
    \end{remark}
    
    \noindent Recall that $M$ is asymptotically harmonic if $H(v)=H$ for all $v \in SM$. In this case we have the following immediate proposition.

    \begin{proposition} \label{corollary for asymptotically harmonic}
        Let $M$ be as in Theorem \ref{strictly convex function} and asymptotically harmonic. Then the spectrum of the Laplacian is purely absolutely continuous and is contained in $[H^2/4,\infty)$.
    \end{proposition}
  
    \begin{remark}
        Observe that for such manifolds, under the assumption of negative sectional curvature, the essential spectrum was shown to be the interval $[H^2/4,\infty)$ in \cite[Theorem 4.4]{CastillonSambusetti:2014}. This result does not rule out, for instance, the existence of eigenvalues in the essential spectrum. The above proposition rules out the existence of not only  eigenvalues but also of the singularly continuous spectrum, thus giving a complete description of the nature of the spectrum. For a related result, see \cite[Theorem 1.3]{Ballmann:2023}.
    \end{remark}

    \noindent  We now consider the example of Riemannian symmetric spaces of noncompact type and higher rank where the mean curvatures of horospheres depends nontrivially on the direction. Let $M$ be an $n$-dimensional symmetric space of noncompact type i.e., $M$ is a simply connected Riemannian manifold with parallel curvature tensor $R$ \cite[Chapter 2]{Eberlein:1996}. Note that $M$ is necessarily of the form $G/K$ where $G=NAK$ is the Iwasawa decomposition of a noncompact connected semisimple Lie group $G$ with finite centre. The rank of the symmetric space is the dimension of $\frak{a}$, the Lie algebra of $A$. The Killing form on the Lie algebra of $G$ induces an inner product on $\frak{a}$ denoted by $\ip{\cdot}{\cdot}$. Let us fix a base point $o=eK \in G/K$. If the rank of $M$ is atleast two, then its sectional curvature vanishes on the $G$- conjugates of $A.o$ in $G/K$. These are the maximal totally geodesic flat submanifolds in $G/K$. For $v \in S_oM$, consider the geodesic $\gamma_{_v}$ satisfying $\gamma_{_v}'(0)=v$ and let $R_v$ be the curvature tensor restricted to $\gamma_{_v}$. Choose an orthonormal basis $\{e_1,\cdots,e_{n-1}\}$ of eigenvectors  of $R_v$ for the subspace in $T_oM$ orthogonal to $v$  i.e.,
	\begin{equation} \label{eigenvalues of curvature operator}
		R_v e_i = - R(v,e_i)v=\lambda_i(v)e_i.
	\end{equation}
    From the structure theory for noncompact connected real semisimple Lie algebras, we have the set of all positive roots denoted by $\sum^{+}$. Let $m_{\alpha}$ denote the multiplicity of $\alpha \in \sum^{+}$. The finite reflection group $W$ generated by the roots is known as the Weyl group. We denote the tangent vector that is dual to $2\rho=\sum\limits_{\alpha \in \sum^{+}}m_{\alpha} \alpha$ also by $2\rho$. The arguments in \cite[Theorem 1]{Eschenburg:1980} show that for each $v \in S_oM$, the mean curvature $\Delta b_v$ is the constant given by $\sum\limits_i\sqrt{\lambda_i(v)}=\ip{v}{2\rho}=:H(v)$. It follows that $H = \sup\limits_{v \in S_oM}H(v)= \norm{2\rho}$
    (see also \cite[Lemma 1]{Wang:2015}). It follows from Theorem \ref{spectrum of asymptotically harmonic manifolds} that the spectrum is contained in $[\norm{\rho}^2,\infty)$. It can be seen from Theorem \ref{cheeger and bottom} below that $\lambda_0(M)=\norm{\rho}^2$ and hence that $H$ is the Cheeger's constant for $M$. \\
        
    \noindent The following Jacobi field arguments shows that $M$ has no focal points. Observe, for example, from \cite{Wang:2015}, that along a geodesic $\gamma_{_v}$ the Jacobi field $Y_i$ satisfying the initial condition $Y_i(0) =0$ and $Y_{i}'(0) = e_i$ is given by the explicit formula
        \begin{equation*}
            Y_i(t)=\frac{1}{\sqrt{\lambda_i(v)}}\sinh\left(\sqrt{\lambda_i(v)}t\right)E_i(t)
        \end{equation*}
        where $E_i$ is a parallel orthonormal frame field along $\gamma_{_v}$ with $E_i(0)=e_i$. Then we see that
        \begin{equation*}
            Y(t)=\sum\limits_{i=1}^{n}\ip{w}{e_i}Y_i(t)
        \end{equation*}
    solves the Jacobi equation \eqref{Jacobi field equation} with initial conditions $Y(0)=0$ and $Y'(0)=w$ for arbitrary $w \in SM$. That $\norm{Y(t)}$ is strictly increasing can then be deduced from the fact that $\sinh$ is an increasing function and so, it follows from \cite[Proposition 4]{Sullivan:1974} that $M$ has no focal points. It is known that $M$ is also Einstein with negative Ricci curvature \cite[\S 2.14,  p. 96]{Eberlein:1996}. Since rank one symmetric spaces of noncompact type are asymptotically harmonic Theorem \ref{spectrum of asymptotically harmonic manifolds} and Proposition \ref{corollary for asymptotically harmonic} holds in this case.   \\  
        
    \noindent Consequently, an application of Theorem \ref{spectrum of asymptotically harmonic manifolds} yields the following classical result. 
    \begin{corollary}
        $M$ be a Riemannian symmetric space of noncompact type. Then the spectrum is purely absolutely continuous and $\lambda_0(M)=\norm{\rho}^2$.
    \end{corollary}
         
    \noindent Equality of the spectrum i.e., $\sigma(\Delta_M)=[\norm{\rho}^2,\infty)$ can be shown using estimates on the $\textbf{c}$-function found in \cite[Chapter 3, p. 101, Theorem 6.3.4]{GangolliVaradarajan:1988} analogous to the case of harmonic manifolds below.\\
     
    \subsection{Spectrum of the Laplacian in harmonic manifolds} 
     
    Recall from the introduction that simply connected noncompact harmonic manifolds are Einstein with Ricci curvature $\kappa \leq 0$. Since $\kappa = 0$ implies the harmonic manifold is Euclidean, we confine ourselves to those with $\kappa < 0$. Proposition \ref{corollary for asymptotically harmonic} applies to  give the following.
    
	\begin{corollary}\label{harmabs}
	Let $M$ be a simply connected noncompact harmonic manifold $M$ without focal points. Then the spectrum of the Laplacian is purely absolutely continuous and is contained in $[H^{2}/4,\infty)$ where $H$ is the constant mean curvature of the horospheres in $M$.
	\end{corollary}

    \noindent Here our interest is to give a novel proof that the spectrum is actually equal to the interval $[H^2/4,\infty)$. For this, we make use of the recently proved \cite{Biswas:2021} unitarity of the Fourier transform for radial functions on harmonic manifolds of purely exponential volume growth, certain estimates on the $c$-function, and the characterisation of harmonic manifolds that the radialisation operator commutes with the Laplacian to pass from radial $\mathrm L^2$- functions to all $\mathrm L^2$- functions. We note that unitarity there is shown only in  dimensions greater than $5$. However, the only simply connected noncompact (nonflat) harmonic manifolds of dimension less than equal to $5$ are rank one Riemannian symmetric spaces of noncompact type \cite[p.145]{Biswas:2021} and this case is covered by the discussion on symmetric spaces. \\
     
    \noindent First, we observe from \cite[Theorems 32, 45]{Knieper:2016} that, noncompact harmonic manifolds without focal points fall into the class of noncompact harmonic manifolds of purely exponential volume growth. Let $o$ be an arbitrary but fixed point in $M$.  A function $f$ on a Riemannian manifold $(M,g)$ is said to be radial around a point $o$ if $f$ is constant on geodesic spheres $S_o(r)$ of radius $r$ centred at $o$. $M$ is said to be harmonic if the volume density in normal coordinates centred at $o$ is a radial function. Given a continuous function $f$, its radialisation about $o$ given by
        \begin{equation*}
             R_o(f)(z) : =\int\limits_{S_o(r)}f(y){\rm{d}}\sigma^r(y)
        \end{equation*}
    where $\sigma^r$ denotes surface area measure on $S_o(r)$ (induced from the Riemannian metric on $M$), normalized to have mass one. Note that for $f$ radial around $o$, we have that $R_of=f$. Harmonicity of the manifold can be seen to be equivalent to the fact that the radialisation operator $R_o$ commutes with the Laplacian $\Delta$. Then for $f \in C_{c}^{\infty}(M)$ which is radial about $o$, the spherical Fourier transform $\F^of$ is given by
         \begin{equation*}
             \F^o(f) (\lambda, v)=\int\limits_{M}f(y)(\phi_{\lambda} \circ d_o)(y) {\rm{d}}y
         \end{equation*}
    for $\lambda \in \C$ and $v \in S_o M$. Here $\phi_{\lambda}$ is the unique function on $[0,\infty)$ with $\phi_{\lambda}(0)=1$ and satisfying $\Delta(\phi_{\lambda}\circ d_o)=(\lambda^2+\frac{H^2}{4})(\phi_{\lambda} \circ d_o)$, where $d_o$ denotes the distance from $o$. Let ${\rm{L}}_{o}^{2}(M)$ denote the closed subspace of ${\rm{L}}^2(M)$ consisting of those functions which are radial about $o$ in ${\rm{L}}^{2}(M)$. It is known from \cite[Theorem 4.7]{Biswas:2021} that $\F^o$ is a unitary operator from ${\rm{L}}_{o}^{2}(M)$ onto ${\rm{L}}^2([0,\infty),C_0|c(\lambda)|^{-2}{\rm{d}}\lambda)$. Here the $c$-function coincides with the Harish Chandra $c$-function on $Im(\lambda) \leq 0$, for rank one symmetric spaces of noncompact type. \\
        
    \noindent We recall that for a real-valued measurable function $f$ on a $\sigma$- finite measure space $(X,\mu)$, the multiplication operator $M_f$ on ${\rm{L}}^2(X,\mu)$ is self-adjoint with domain $\mathcal{D}(M_f):=\{g \in {\rm{L}}^2(X,\mu) : M_f g \in {\rm{L}}^2(X,\mu) \}$. Let $m$ be the function on $[0,\infty)$ defined by $m(\lambda)=\lambda^2+\frac{H^2}{4}$. Then we see that the multiplication operator $M_m$ on ${\rm{L}}^2([0,\infty),C_0|c(\lambda)|^{-2}{\rm{d}}\lambda)$ is self-adjoint on 
        \begin{equation*}
            \mathcal{D}(M_m)=\{f \in {\rm{L}}^2([0,\infty),C_0|c(\lambda)|^{-2}{\rm{d}}\lambda) : M_m f \in  {\rm{L}}^2([0,\infty),C_0|c(\lambda)|^{-2}{\rm{d}}\lambda) \}.    
        \end{equation*}
    Moreover, the spectrum of a multiplication operator $M_m$ is equal to the essential range of $m$ \cite[Theorem 4.5]{Borthwick:2020}. Here
        \begin{equation*}
            ess-range(m):=\{x \in \R: |m^{-1}(x-\epsilon,x+\epsilon))|>0, ~ \forall ~ \epsilon > 0 \}.
        \end{equation*}
    where $|\cdot|$ denotes the measure. Now to determine the spectrum of $M_m$ we argue as follows. For any $x \in \R$ and $\epsilon > 0$,
        \begin{equation*}
            m^{-1}(x-\epsilon,x+\epsilon)=\{\lambda \in [0,\infty):x-\epsilon< \lambda^2+\frac{H^2}{4} < x+\epsilon\}.
        \end{equation*}
    Note that when $x< \frac{H^2}{4}$ and $\epsilon = -(x - \frac{H^2}{4})$, $m^{-1}(x-\epsilon,x+\epsilon)$ is empty and so $x \notin ess-range(m)$. Equivalently, $ess-range(m) \subseteq [\frac{H^2}{4},\infty)$. Now for any $x \in [\frac{H^2}{4},\infty)$ and $\epsilon>0$ we have, 
        \begin{equation}\label{essential range}
        m^{-1}(x-\epsilon,x+\epsilon)=
            \begin{cases}
                \left[0,\sqrt{x-\frac{H^2}{4}+\epsilon} \right) & \text{if} ~ x-\frac{H^2}{4}-\epsilon < 0 \\ \\
                \left[\sqrt{x-\frac{H^2}{4} - \epsilon}, \sqrt{x-\frac{H^2}{4}+\epsilon} \right) & \text{if} ~ x-\frac{H^2}{4}-\epsilon \geq 0. \\
             \end{cases}
        \end{equation}
    We recall the estimates for $|c(\lambda)|^{-1}$ given in \cite[\S 4.1]{Biswas:2021}.\\
       
    \noindent If $x-\frac{H^2}{4}-\epsilon < 0$, then $\lambda \in \left[0,\sqrt{x-\frac{H^2}{4}+\epsilon} \right)$. If  $\sqrt{x-\frac{H^2}{4}+\epsilon} \leq K$ then we have 
        \begin{equation*}
             \frac{1}{C}\lambda \leq |c(\lambda)|^{-1} 
        \end{equation*}
    and so
         \begin{equation*}
            \int\limits_{ m^{-1}(x-\epsilon,x+\epsilon)}C_0|c(\lambda)|^{-2}{\rm{d}}\lambda \geq  \frac{C_0}{C}\int\limits_{0}^{\sqrt{x-\frac{H^2}{4}+\epsilon}}\lambda^2{\rm{d}}\lambda > 0,\quad \text{since} ~ x-\frac{H^2}{4}+\epsilon >0.
        \end{equation*}
    Now, if $K \in \left(0,\sqrt{x-\frac{H^2}{4}+\epsilon}\right)$ then
        \begin{equation*}
            \begin{cases}
                \frac{1}{C}\lambda \leq |c(\lambda)|^{-1} , & \lambda \in  [0,K) \\
                  \frac{1}{C}\lambda^{\frac{n-1}{2}} \leq |c(\lambda)|^{-1} , & \lambda \in  \left[K,\sqrt{x-\frac{H^2}{4}+\epsilon} \right)
            \end{cases}
        \end{equation*}
    holds thereby
        \begin{equation*}
            \int\limits_{ m^{-1}(x-\epsilon,x+\epsilon)}C_0|c(\lambda)|^{-2}{\rm{d}}\lambda \geq 
            \begin{cases}
                \displaystyle{\frac{C_0}{C}\int\limits_{0}^{K}\lambda^2{\rm{d}}\lambda} , & \lambda \in [0,K) \\ \\
                 \displaystyle{ \frac{C_0}{C}\int\limits_{K}^{\sqrt{x-\frac{H^2}{4}+\epsilon}}\lambda^{n-1}{\rm{d}}\lambda }, & \lambda \in \left[K,\sqrt{x-\frac{H^2}{4}+\epsilon} \right) 
            \end{cases}
        \end{equation*}
    and hence the integral is always strictly positive. A similar argument shows that the integral remains strictly positive for the other case in \eqref{essential range}. It follows that $ess-range(m)=[\frac{H^2}{4},\infty)$ and so the spectrum of $M_m$ is equal to $[\frac{H^2}{4},\infty)$. For functions that are not necessarily radial, a notion of Fourier transform, relative to a point, was introduced in \cite[\S 5]{Biswas:2021} where it was shown to be an isometry but the unitarity is not known. This Fourier transform coincides with the spherical Fourier transform $\F^o$ for radial functions \cite[Proposition 5.3]{Biswas:2021}. Then we see as in \cite[Lemma 6.5]{Brammen:2024}, that
        \begin{equation*}
            \F^o\Delta f = M_m \F^o f, \quad f \in {\rm{L}}_{o}^{2}(M).
        \end{equation*}
    Since $\F^o$ is unitary and $\mathcal{D}(M_m)=\F^o(\mathcal{D}(\Delta) \cap {\rm{L}}_{o}^{2}(M))$, we have the following lemma.
                
    \begin{lemma} \label{spectrum of radial laplacian}
        The spectrum of the Laplacian $\Delta$ on $\mathcal{D}(\Delta) \cap {\rm{L}}_{o}^{2}(M)$ is equal to $[H^2/4,\infty)$. 
    \end{lemma}

    \noindent We now claim that,
    \begin{lemma}\label{inclusion of spectrum}
        $[H^2/4,\infty) \subseteq \sigma(\Delta)$ where $\sigma(\Delta)$ is the spectrum of $\Delta$ with domain $\mathcal{D}(\Delta)$.
    \end{lemma}
    \begin{proof}
    If $\mu$ is not in $\sigma(\Delta)$, then $\Delta-\mu:\mathcal{D}(\Delta) \to {\rm{L}}^{2}(M)$ is invertible with bounded inverse. In particular, $\Delta - \mu$ is injective on $\mathcal{D}(\Delta) \cap {\rm{L}}_{o}^{2}(M)$.  If $f \in {\rm{L}}_{o}^{2}(M)$, then there exists a unique $g$ in $\mathcal{D}(\Delta)$ such that $(\Delta-\mu)g=f$. We have
        \begin{equation*}
          (\Delta-\mu)g = f  = R_o f = R_o(\Delta -\mu)g = (\Delta-\mu)R_og 
        \end{equation*}
    and since $\Delta - \mu$ is invertible we have $R_o g =g$ i.e., $\Delta-\mu$ is surjective from $\mathcal{D}(\Delta) \cap {\rm{L}}_{o}^{2}(M)$ onto ${\rm{L}}_{o}^{2}(M)$. Since $(\Delta-\mu)^{-1}$ is bounded on ${\rm{L}}^2(M)$ and ${\rm{L}}_{o}^{2}(M)$ is a closed subspace, $(\Delta-\mu)^{-1}$ is also bounded on $\mathcal{D}(\Delta)\cap {\rm{L}}_{o}^{2}(M)$. The claim now follows from Lemma \ref{spectrum of radial laplacian}.
        \end{proof}
        
    \noindent The equality of the spectrum is obtained from Corollary \ref{harmabs} and Lemma \ref{inclusion of spectrum}. In summary, we have the following proposition.
        
 	\begin{proposition}\label{harspec}
 	    Let $M$ be a simply connected noncompact harmonic manifold without focal points. Then the spectrum of Laplacian is purely absolutely continuous and is equal to $[H^2/4,\infty)$.
 	\end{proposition}

    \section{Isoperimetric constant and volume entropy}

   \noindent We consider here more general class of manifolds without conjugate points than in the preceding sections. In \S 3, we studied how the volumes of geodesic balls evolve with increasing radii. The exponential growth of volumes of geodesic balls in $M$ is measured by the volume entropy of $M$, defined by
          \begin{equation}\label{volume entropy}
             h_{vol}(M):= \limsup\limits_{r \to \infty}\frac{\log(\mathrm{Vol}(B_o(r)))}{r}
         \end{equation}
    Triangle inequality implies that $h_{vol}(M)$ does not depend on the choice of centres of the balls and hence is an invariant of $M$. For manifolds of infinite volume, the following fundamental inequality relating the volume entropy and the bottom of the essential spectrum  was shown in \cite[Theorem 1]{Brooks:1981}:
    \begin{equation*}
        \lambda_{0}^{ess}(M) \leq h_{vol}^2(M)/4.
    \end{equation*}
    In view of Cheeger's inequality, we therefore have 
    \begin{equation*}
        h_{Cheeger}^{2}/4 \leq \lambda_0(M) \leq \lambda_{0}^{ess}(M) \leq h_{vol}^2(M)/4.
    \end{equation*}
    Each of these inequalities can be strict and so it is interesting to study when equality holds in all of them. It has been shown that $h_{Cheeger}(M) = h_{vol}(M)$, for example, if $M$ is asymptotically harmonic and has bounded asymptote \cite[Theorem 5.4]{KnieperPeyerimhoff:2015} or if $M$ is a symmetric space of noncompact type \cite[Theorem 1]{Wang:2015}. We address this question below in a setting of greater generality which includes the above classes of examples. Recall from \eqref{supremum of mean curvatures} that $H \leq h_{Cheeger}(M)$ for manifolds of continuous asymptote in which horospheres corresponding to each $v \in S_oM$ all have the same constant mean curvature $H(v)$. We now proceed to show equality of Cheeger's constant and the volume entropy if $M$ is with bounded asymptote.

    \begin{theorem}\label{cheeger and bottom}
         Let $M$ be a manifold with bounded asymptote with bound $k$ as in \eqref{bounded asymptote}. Suppose there exists $o \in M$ such that the horospheres corresponding to each $v \in S_oM$ all have the same constant mean curvature denoted by $H(v) \geq 0$ with $H=\sup\limits_{v\in S_oM}H(v)$. Then we have
         \begin{itemize}
             \item [(i)]  $h_{Cheeger}(M)=H=h_{vol}(M)$. 
             \item[(ii)]  $\sigma(\Delta) \subset [H^2/4,\infty)$ with $\lambda_0(M)=\lambda_{0}^{ess}(M) = H^2/4$. 
         \end{itemize}
     \end{theorem} 

    \begin{proof} We first show that the volume of $M$ is infinite. For $v\in S_oM$, consider the Lagrange tensor $A_v$ along $\gamma_{_v}$ given by the initial conditions $A_v(0)=0$ and $A_{v}'(0)=\mathrm{Id}$ . Since $A_v$ satisfies the Jacobi equation \eqref{Jacobi tensor equation}, its Taylor expansion about $t=0$ yields
    \begin{equation*}
        A_v(t) = A_{v}(0)+tA_{v}'(0)+\frac{t^2}{2}A_{v}''(0)+o(t^2),\quad \text{as}~ t \to 0^+.
    \end{equation*}
    The Jacobi equation and the initial conditions yield
    \begin{equation*}
          A_v(t) = t \left( \mathrm{Id}+o(t)\right) \quad \text{as}~ t \to 0^+.
    \end{equation*}
    By continuity, there exists $\eta >0$ such that $\det(A_v(t)) > 0$ for $0<t < \eta$. As $\gamma_{_v}$ contains no conjugate points, the Lagrange tensor $A_v(t)$ is nonsingular for $t \neq 0$, and so $\det(A_v(t)) > 0$ for all $t >0$. From \cite[proof of Proposition 2]{Eschenburg:1977} we have for any $s > 0$
   \begin{equation}\label{comparing mean curvature of sphere and horosphere}
        (-D_{s,-\phi^{-t}(v)}'(0)) - (-D_{-\phi^{-t}(v)}'(0)) > 0,
   \end{equation} 
   and since $\mathrm{trace}(-D_{-\phi^{-t}(v)}'(0))= H(-v)$ for each $t > 0$, using \cite[p. 151, 155]{Knieper:2016}
   \begin{equation}\label{determinant of Lagrange tensor}
   \det(A_v(t))'  = \mathrm{trace}(A_{v}'(t)A_{v}^{-1}(t))\det(A_v(t)) = \mathrm{trace}(-D_{t,-\phi^{-t}(v)}'(0))\det(A_v(t)) > 0.
   \end{equation}
   It follows that for each $v \in S_oM$, $\det(A_v(t))$ is monotonically increasing and so for $0 < t_1 < t_ 2$
   \begin{equation*}
       \mathrm{Area}(S_o(t_1))= \int\limits_{S_oM}\det(A_v(t_1))d\theta_o(v) < \int\limits_{S_oM}\det(A_v(t_2))d\theta_o(v) =\mathrm{Area}(S_o(t_2)).
    \end{equation*}
    Now for $r \gg 1$,
   \begin{equation*}
       \mathrm{Vol}(B_o(r))  = \int\limits_{0}^{r}\mathrm{Area}(S_o(t)) dt > \int\limits_{1}^{r}\mathrm{Area}(S_o(t))dt > \int\limits_{1}^{r}\mathrm{Area}(S_o(1)) dt = \mathrm{Area}(S_o(1)) (r-1) 
   \end{equation*}
    and therefore, the volume of $M$ is infinite. We remark here that the arguments still holds if we replace manifold with bounded asymptote by those of continuous asymptote. \\
    
    \noindent Since $\mathrm{Vol}(M)$ is infinite, if $h_{vol}(M) \leq H$, then we have from  \cite[Theorem 1]{Brooks:1981} that 
     \begin{equation*}
        \lambda_{0}^{ess}(M) \leq h_{vol}^2(M)/4 \leq H^2/4.
     \end{equation*}
     It follows from \eqref{cheeger inequality} and \eqref{supremum of mean curvatures} that $h_{Cheeger}(M)=H=h_{vol}(M)$, $\lambda_0(M)=\lambda_{0}^{ess}(M)=H^2/4$ and consequently that $\sigma(\Delta_M)\subseteq [H^2/4, \infty)$. \\
     
     \noindent So it only remains to show that $h_{vol}(M) \leq H$. To this end, we first make the following observation from calculus \cite[Theorem II]{Taylor:1952}. Let $f$ and $g$ be real-valued differentiable functions on $(0,\infty)$. Suppose that
    \begin{itemize}
        \item [(a)] $g(x), ~ g'(x) \neq 0$ for all $x \in (0,\infty)$
        \item[(b)] $\lim\limits_{x \to \infty}|g(x)|=\infty$.
    \end{itemize}
    Then 
    \begin{equation}\label{lhopital}
         \limsup\limits_{x \to \infty}\frac{f(x)}{g(x)} \leq \limsup\limits_{x \to \infty}\frac{f'(x)}{g'(x)}
    \end{equation}
    Invoking \eqref{lhopital} for the following pairs of functions
   \begin{itemize}
       \item [(i)] $f(r)=\log(\mathrm{Vol}(B_o(r)))$ and $g(r)=r$ defined on $(0,\infty)$ and
       \item [(ii)] $f(r)=\mathrm{Area}(S_o(r))$ and $g(r)=\mathrm{Vol}(B_o(r))$ defined on $(0,\infty)$. 
   \end{itemize}
   we have that
   \begin{align*}
        h_{vol}(M) & = \limsup\limits_{r \to \infty}\frac{\log(\mathrm{Vol}(B_o(r)))}{r}  \leq \limsup\limits_{r \to \infty}\frac{\mathrm{Area}(S_o(r))}{\mathrm{Vol}(B_o(r))} = \limsup\limits_{r \to \infty} ~ \frac{\displaystyle{\int\limits_{S_oM}\det(A_v(r))d\theta_o(v)}}{\displaystyle{\int\limits_{0}^{r}\int\limits_{S_oM}\det(A_v(t))d\theta_o(v)\mathrm{d}t}}  \\
        & \overset{\eqref{determinant of Lagrange tensor}} \leq  \limsup\limits_{r \to \infty} ~ \frac{\displaystyle{\int\limits_{S_oM}\mathrm{trace}(A_{v}'(r)A_{v}^{-1}(r))\det(A_v(r))d\theta_o(v)}}{\displaystyle{\int\limits_{S_oM}\det(A_v(r))d\theta_o(v)}}  \\  
        & \overset{\eqref{determinant of Lagrange tensor}} = \limsup\limits_{r \to \infty} ~ \frac{\displaystyle{\int\limits_{S_oM}\mathrm{trace}(-D_{r,-\phi^r(v)}'(0))\det(A_v(r))d\theta_o(v)}}{\displaystyle{\int\limits_{S_oM}\det(A_v(r))d\theta_o(v)}} 
   \end{align*}
   \begin{align}\label{uniform estimate}
         & \leq \limsup\limits_{r \to \infty} ~ \frac{\displaystyle{\int\limits_{S_oM}\left[\mathrm{trace}(-D_{r,-\phi^r(v)}'(0)) - H(-v)\right]}\det(A_v(r))d\theta_o(v)}{\displaystyle{\int\limits_{S_oM}\det(A_v(r))d\theta_o(v)}} \nonumber \\
         & \hspace*{20mm} + \limsup\limits_{r \to \infty} ~ \frac{\displaystyle{\int\limits_{S_oM} H(-v)}\det(A_v(r))d\theta_o(v)}{\displaystyle{\int\limits_{S_oM}\det(A_v(r))d\theta_o(v)}} \nonumber \\
        & \leq \limsup\limits_{r \to \infty} ~ \frac{\displaystyle{\int\limits_{S_oM}\left[\mathrm{trace}(-D_{r,-\phi^r(v)}'(0)) - H(-v)\right]}\det(A_v(r))d\theta_o(v)}{\displaystyle{\int\limits_{S_oM}\det(A_v(r))d\theta_o(v)}} + H. 
    \end{align}
    Observe that for each $v\in S_oM$ and $r > 0$, we have $\mathrm{trace}(-D_{-\phi^r(v)}'(0))= H(-v)$ and so using \cite[Proposition 5]{Eschenburg:1977} we have
    \begin{align*}
       \mathrm{trace}(-D_{r,-\phi^r(v)}'(0)) - \mathrm{trace}(-D_{-\phi^r(v)}'(0)) & \leq (n-1)\norm{(-D_{r,-\phi^r(v)}'(0)) - (-D_{-\phi^r(v)}'(0))} \\
        & \leq \frac{(n-1) k^2}{r}.
    \end{align*}
    It follows from \eqref{uniform estimate} that $h_{vol}(M) \leq H$, concluding the proof of the theorem. 
    \end{proof}

\section{Funding}

\noindent This work was supported by the Centre for Operator Algebras, Geometry, Matter and Spacetime, Ministry of Education, Government of India through Indian Institute of Technology Madras [SB22231267MAETWO008573 to A.P. and S.B.]; and by the National Board for Higher Mathematics, Department of Atomic Energy, Government of India [Ref.No:0203/11/2019-R\&D-II to S.B.]

\bibliography{references}
\bibliographystyle{amsalpha}

\end{document}